\documentclass[reqno,12pt]{amsart}

\usepackage{amsmath}
\usepackage{amssymb}
\usepackage[pdftex]{graphicx}
\usepackage[pdftex,colorlinks,citecolor=blue]{hyperref}
\usepackage{cleveref}
\usepackage{bbm} %% Font for 1 / indicator

\makeatletter
\def\@marginparreset{%
  \reset@font
  \tiny
  \linespread{0.9}
  \@marginparfont
  \@setminipage
}
\def\@marginparfont{\normalfont\itshape\color{red}\raggedleft}
\makeatother

\crefname{theorem}{Theorem}{Theorems}
\crefname{lemma}{Lemma}{Lemmas}
\crefname{remark}{Remark}{Remarks}
\crefname{prop}{Proposition}{Propositions}
\crefname{defn}{Definition}{Definitions}
\crefname{corollary}{Corollary}{Corollaries}
\crefname{conjecture}{Conjecture}{Conjectures}
\crefname{chapter}{Chapter}{Chapters}
\crefname{section}{Section}{Sections}
\crefname{figure}{Figure}{Figures}

\theoremstyle{plain}
\newtheorem{theorem}{Theorem}[section]
\newtheorem{lemma}[theorem]{Lemma}

\theoremstyle{definition}

\theoremstyle{remark}

\newcommand{\cF}{\mathcal F}
\newcommand{\cG}{\mathcal G}
\newcommand{\cH}{\mathcal H}

\newcommand{\cP}{\mathcal P}

\renewcommand{\subset}{\subseteq}
\renewcommand{\supset}{\supseteq}
\renewcommand{\emptyset}{\varnothing}

\newcommand{\gen}[1]{\langle #1 \rangle}
\newcommand{\PP}{\mathbb{P}}
\newcommand{\EE}{\mathbb{E}}
\newcommand{\IND}{\mathbbm{1}}

\newcommand{\cost}{\textrm{cost}}
\newcommand{\e}{\mathrm{e}}
\newcommand{\pc}{p_\mathrm{c}}
\newcommand{\qc}{q_\mathrm{c}}
\newcommand{\qf}{q_\mathrm{f}}
\renewcommand{\le}{\leqslant}
\renewcommand{\ge}{\geqslant}

\usepackage[makeroom]{cancel}

\begin{document}

\title{Thresholds and expectation thresholds for larger $p$}

\date{\today}

\author[T. Przyby\l owski]{Tomasz Przyby\l owski}
\thanks{First author was supported by the Additional Funding Programme for Mathematical Sciences, delivered by EPSRC (EP/V521917/1) and the Heilbronn Institute for Mathematical Research.}

\author[O. Riordan]{Oliver Riordan}
\address{Mathematical Institute, University of Oxford \\
 Oxford OX2 6GG, UK}
\email{przybylowski@maths.ox.ac.uk, riordan@maths.ox.ac.uk}

\begin{abstract}
Let $\pc$ and $\qc$ be the threshold and the expectation threshold, respectively, of an increasing family $\cF$ of subsets of a finite set $X$, and let $l$ be the size of a largest minimal element of $\cF$. Recently, Park and Pham proved the Kahn--Kalai conjecture, which says that $\pc \leqslant K q_c \log_2 l$ for some universal constant $K$. Here we slightly strengthen their result by showing that $p_c \leqslant 1 - \e^{-K \qc \log_2 l}$. The idea is to apply the Park--Pham Theorem to an appropriate `cloned' family $\cF_k$, reducing the general case (of this and related results) to the case where the individual element probability $p$ is small.

\end{abstract}

\maketitle

\section{Introduction}

Given a finite set $X$ and $p \in [0,1]$, let $X_p$ be the random subset where each element is included with probability $p$, independently of the others; we call this a \emph{$p$-random subset of $X$} for short. A family $\cF \subset \cP(X)$ is \textit{increasing} (or an \emph{up-set}) if $A \in \cF$ and $A \subset B \subset X$ implies $B \in \cF$. Throughout, $\cF$ will denote a non-trivial increasing family, so $\cF \neq \emptyset$, $\cP(X)$. 

Let $\cG \subset \cP(X)$. Following (to some extent) the terminology in \cite{parkpham}, for $q \in [0,1]$, we define the \textit{$q$-cost} of $\cG$ to be
\begin{equation*}
\cost_q(\cG) := \sum_{S \in \cG} q^{|S|} = \EE[ | \{S \in \cG\colon S \subset X_q\} |].
\end{equation*}
We say that $\cG$ is \textit{$q$-cheap} if $\cost_q(\cG) \le \frac{1}{2}$. We say that $\cG$ \emph{covers} $\cF$ if 
\begin{equation*}
\cF \subset \gen{\cG} := \bigcup_{S \in \cG} \{T\colon T \supset S\},
\end{equation*}
i.e., if $\cF$ is contained in the up-set generated by $\cG$,
and that $\cG$ is a \textit{$q$-cheapest} cover of $\cF$ if it has minimal $q$-cost among all covers of $\cF$. Of course, such a $\cG$ may not be unique.

The \textit{threshold} of $\cF$ is the unique $\pc \in (0,1)$ such that $\PP(X_{\pc} \in \cF) = \frac{1}{2}$. Following Kahn and Kalai~\cite{kahnkalai}, the \textit{expectation threshold} of $\cF$ is defined to be
\begin{equation*}
\qc(\cF) := \sup\{ q \in [0,1]\colon \text{there exists a $q$-cheap cover of } \cF\}.
\end{equation*}
Kahn and Kalai noted that $\qc$ easily (by applying the union bound) gives a lower bound on~$\pc$. Strikingly, they conjectured, and Park and Pham \cite{parkpham} proved, that a not-too large multiple of $\qc$ provides an upper bound on $\pc$. Let $l(\cF)$ be the size of a largest minimal element of $\cF$.

\begin{theorem}[\cite{parkpham}]\label{parkpham}
There is a universal constant $K$ such that for every finite set~$X$ and non-trivial increasing family $\cF \subset \cP(X)$ with $l(\cF) \ge 2$,
\begin{equation}\label{parkphameq}
\qc(\cF) \le \pc(\cF) \le K \qc(\cF) \log_2 l(\cF).
\end{equation}
\end{theorem}
 
In~\cite{bell}, a slightly simpler version of the proof from~\cite{parkpham} is given, which yields $K \le 16$ (or $K \le 11$ if $l(\cF) \ge 4$). The arguments in~\cite{parkpham, bell} easily give the stronger bound
\begin{equation}\label{eqpp2}
 \pc(\cF) \le 1 - (1-K \qc(\cF))^{\log_2 l(\cF)},
\end{equation}
whenever $\qc(\cF)\le 1/K$. Our aim is to strengthen this bound even further.

\begin{theorem}\label{ourthm}
Let $K$ be the universal constant from \cref{parkpham}. Then for every finite set~$X$ and non-trivial increasing family $\cF \subset \cP(X)$ with $l(\cF) \ge 2$,
\begin{equation}\label{ourbnd}
\pc(\cF) \le 1 - \e^{-K \qc(\cF) \log_2 l(\cF)}.
\end{equation}
\end{theorem}
Since $\e^{-xy}\ge (1-x)^y\ge 1-xy$, the bound \eqref{ourbnd} does indeed imply \eqref{eqpp2} and \eqref{parkphameq}. The strategy of the proof is simply to apply Theorem~\ref{parkpham} to an appropriate transformed family~$\cF_k$. The main ingredient is a `cloning' lemma showing that the expectation threshold changes as expected under this transformation.

The improvement in \eqref{ourbnd} comparing with \eqref{eqpp2} is most significant when $K\qc$ is relatively large. In particular, \eqref{eqpp2} gives no information when $K\qc \ge 1$. However, as pointed out by Keith Frankston, in this case neither does \eqref{ourbnd}, because of the trivial bound $\pc \le 2^{-1/l}$ obtained by considering a single minimal element of $\cF$. This trivial bound is strictly stronger than \eqref{ourbnd} already when $l \ge 3$ and $K\qc \ge 1$. Thus it is far from clear that \cref{ourthm} will have any applications. Nevertheless, we hope that the method of proof is interesting, namely using a `cloning' lemma to transfer results from the small $p$ regime to general $p$.
%Thus we do not expect \cref{ourthm} to have any applications, but we hope that the `cloning' lemma might be of independent interest.

\section{The $k$-clone model}
Let $\cF\subset \cP(X)$ be a non-trivial increasing family, and let $k$ be a positive integer. Informally, we define the $k$-cloned version of $\cF$ as follows: instead of performing a single experiment for each $x\in X$ to determine whether $x\in X_p$, we perform $k$ experiments, declaring $x\in X_p$ if and only if at least one of $(x,1),\ldots,(x,k)$ is present in the corresponding random subset $Y_{p'}$ of $Y=X\times [k]$. Then $\cF_k$ will be the set of subsets of $Y$ corresponding to~$\cF$.

Formally, we define the \textit{$k$-cloned} family $\cF_k \subset \cP(X \times [k])$ to be
\begin{equation*}
\cF_k := \{S\subset X\times [k] : \pi(S) \in \cF \},
\end{equation*}
where $\pi\colon X \times [k] \to X$ is the projection onto the first coordinate,
and $\pi(S)$ denotes the image of a set $S$ under $\pi$.
Since $\cF$ is increasing and non-trivial, so is $\cF_k$.

If $Y_p$ is a $p$-random subset of $X \times [k]$, then $\pi(Y_p)$ is a $(1-(1-p)^k)$-random subset of ~$X$. It is thus immediate that
\begin{equation}\label{eqclonep}
\pc(\cF) = 1-(1-\pc(\cF_k))^k.
\end{equation}

If every cheapest cover $\cH$ of $\cF_k$ arose in the natural way (see below) from a cheapest cover of $\cF$, there would be a simple relationship between the expectation thresholds of $\cF$ and $\cF_k$. Although the former does not hold in general, the latter does.

\begin{lemma} \label{lemclone}
Let $k$ be a positive integer and let $\cF$ be a non-trivial increasing family of subsets of a finite set $X$. Then
\begin{equation*}
  \qc(\cF_k) = \frac{1}{k} \qc(\cF).
\end{equation*}
\end{lemma}

\begin{proof}
For the lower bound, let $\cG$ be a $q$-cheap cover of $\cF$. It suffices to construct a ($q/k$)-cheap cover of $\cF_k$. 

Given $S \subset X$ let 
\begin{equation}
 \Psi(S) = \{S' \subset X \times [k]\colon \pi(S') = S \text{ and } |S'| = |S|\} \label{defps}
\end{equation}
be the set of \emph{minimal} pre-images of $S$ under $\pi$.
Note that $|\Psi(S)| = k^{|S|}$, since we must choose one $(x,i)\in X\times [k]$ for each $x\in S$. Let
\begin{equation}
 \cH := \bigcup_{S \in \cG} \Psi(S). \label{defh}
\end{equation}
It is immediate that $\cH\subset \cP(X \times [k])$ is a cover of $\cF_k$. Indeed, if $Y \in \cF_k$ then $\pi(Y) \in \cF$ and hence there is $Z \in \cG$ such that $\pi(Y) \supset Z$. Then $Y \in \gen{\Psi(Z)} \subset \gen{\cH}$. Furthermore,
\begin{equation*}
\cost_{q/k}(\cH) = \sum_{T \in \cH} \left(\frac{q}{k}\right)^{|T|} = \sum_{S \in \cG} \sum_{T \in \Psi(S)} \left(\frac{q}{k}\right)^{|T|} = \sum_{S \in \cG} k^{|S|} \left(\frac{q}{k}\right)^{|S|} = \cost_q(\cG).
\end{equation*}
Since $\cG$ is $q$-cheap, $\cH$ is ($q/k$)-cheap. Hence $\qc(\cF_k) \ge \frac{1}{k} \qc(\cF)$.

Turning to the upper bound, let $\cH$ be an arbitrary $q$-cheap cover of $\cF_k$. It suffices to deduce the existence of a $kq$-cheap cover of $\cF$. For this we use an averaging argument, over all $k^{|X|}$ copies of $X$ living inside $X\times [k]$.

Formally, for $X' \in \Psi(X)$ let $\cH_{X'} = \cP(X') \cap \cH$.
Observe that $\cH_{X'}$ is a cover of $\cP(X') \cap \cF_k$. The projection~$\pi$ induces an inclusion- and size-preserving bijection from $\cP(X')$ to $\cP(X)$. Hence, $\pi(\cH_{X'})$ is a cover of~ $\cF$ and $\cost_r(\cH_{X'}) = \cost_r(\pi(\cH_{X'}))$ for any $r \in [0,1]$.

A set $S\subset X\times [k]$ is contained in \emph{some} $X'\in \Psi(X)$ if and only if $S$ contains no `duplicates' (two elements from the same fibre $\{x\}\times [k])$. In that case, $S$ is contained in exactly $k^{|X|-|S|}$ sets $X'\in \Psi(X)$, since in constructing $X'$ we have a free choice from each $\{x\}\times [k]$ for $x \in X \setminus \pi(S)$.
Thus
\begin{multline*}
\sum_{X' \in \Psi(X)} \cost_{kq}( \cH_{X'} ) = \sum_{X' \in \Psi(X)} \sum_{S \in \cH_{X'}} (kq)^{|S|} = \sum_{S \in \cH} (kq)^{|S|} \sum_{X' \in \Psi(X)} \IND(S \subset X') \nonumber \\
\le \sum_{S \in \cH} (kq)^{|S|} k^{|X|-|S|}  =  k^{|X|} \sum_{S \in \cH} q^{|S|} = k^{|X|}  \cost_q(\cH).
\end{multline*}

Therefore there is some $X' \in \Psi(X)$ such that $\cost_{kq}(\cH_{X'}) \le \cost_q(\cH)$. Then $\pi(\cH_{X'})$ is a $kq$-cheap cover of $\cF$, as desired.
\end{proof}

We are now ready to prove Theorem~\ref{ourthm}.
\begin{proof}[Proof of Theorem~\ref{ourthm}]
Consider the $k$-clone $\cF_k$ of $\cF$ for $k > \frac1{K\qc(\cF)\log_2 l(\cF)}$. Applying the Park--Pham result, Theorem~\ref{parkpham}, to $\cF_k$ and then using Lemma~\ref{lemclone} we have
\begin{equation*}
 \pc(\cF_k) \le K  \qc(\cF_k) \log_2 l(\cF_k) = K \frac{\qc(\cF)}{k} \log_2 l(\cF),
\end{equation*}
since $l(\cF_k) = l(\cF)$. Applying \eqref{eqclonep} we deduce that 
\begin{equation*}
 \pc(\cF) \le  1 - \left(1 - K  \frac{\qc(\cF)}{k} \log_2 l(\cF)\right)^k.
\end{equation*}
Taking the $k \to \infty$ limit yields the result.
\end{proof}

\section{Discussion} \label{disc}
The proof of Theorem~\ref{ourthm} is based on `scaling' the bound \eqref{parkphameq} up from the sparse case. Since we can take our cloning factor $k$ to be arbitrarily large, only the asymptotic behaviour of \eqref{parkphameq} near zero matters. In particular, we obtain exactly the same final result if we start from \eqref{eqpp2} instead. More generally, to deduce \eqref{ourbnd} we only need the bound $\pc(\cF) \le (K+o(1)) \qc(\cF) \log_2 l(\cF)$ as $\qc(\cF)\to 0$ with $l(\cF)$ fixed.

At first sight, it might seem `obvious' that $\qc(\cF_k)=\qc(\cF)/k$, since (one might expect) every $q$-cheapest cover of $\cF_k$ arises from a $q$-cheapest cover of $\cF$ via the cloning procedure in \eqref{defps} and \eqref{defh}. However, this is not true in the general case. For $\cF = \{\{1,2,3\}, \{1,2\}, \{1,3\}\}$ we have $\qc(\cF_2) = 1/4$ and $\cF_2$ admits a $1/4$-cheapest cover not arising in this way.

Related to this, one might expect that if $\cF$ is symmetric under the action of some group~$G$ on the ground-set $X$, then for every $q \in [0,1]$ it has a $q$-cheapest cover which has the same symmetry. This again is not true in general.

Finally, we note that obvious modifications to the proof of \cref{lemclone} show that the \textit{fractional expectation threshold} $\qf$ defined in \cite{frankston,talagrand}, which behaves better with respect to symmetry, scales under cloning in the same way that $\qc$ does: we have $\qf(\cF_k) = \qf(\cF) / k$.

Since the first version of the current paper was written, B.~Park and J.~Vondr\'ak~\cite{parkvondrak} have proved a non-uniform version of the Park--Pham theorem. They suggest that the `cloning' procedure described above gives an alternative proof of their result, but with worse constants. However, \emph{starting} from their Theorem 1, one can apply the cloning procedure (adapted in the natural way to the non-uniform setting) to deduce the (very slightly) stronger bound $p_x = 1 - \e^{-\qc (4\lfloor \log_2(2l) \rfloor + 7)}.$ In other words, our reduction can be combined with their result also.

\end{document}